\documentclass{amsart}
\usepackage[colorlinks]{hyperref}
\usepackage{hyperref}
\usepackage[colorinlistoftodos]{todonotes}
\DeclareMathOperator{\lcm}{lcm}

\usepackage{float}

\usepackage{amsxtra}
\usepackage{amsthm}
\usepackage{amsmath}
\usepackage{bbm}
\usepackage{cancel}
\usepackage{verbatim}
\usepackage{amsrefs}
\usepackage{geometry}
\usepackage{amssymb, amsfonts}
\usepackage{cite}
\usepackage{graphicx}
 
\theoremstyle{plain}
\newtheorem{Thm}{Theorem}

\newtheorem{Lem}[Thm]{Lemma}

\newtheorem{Def}[Thm]{Definition}

\theoremstyle{definition}

\newtheorem{Ex}[Thm]{Example}

\renewcommand{\bar}{\overline}

\newcommand{\C}{\mathbb C}

\newcommand{\Mod}[1]{\ (\mathrm{mod}\ #1)}

\numberwithin{equation}{section}

\begin{document}
\title{Invariant CR Mappings between Hyperquadrics}

\author[Dusty Grundmeier]{Dusty Grundmeier}
\address{Department of Mathematics \\
Harvard University}
\email{deg@math.harvard.edu}

\author[Kemen Linsuain]{Kemen Linsuain}
\address{Department of Mathematics \\
Harvard University}
\email{klinsuain@alumni.harvard.edu}

\author[Brendan Whitaker]{Brendan Whitaker}
\address{Department of Mathematics \\
Ohio State University}
\email{whitaker.213@osu.edu}

\thanks{{\em 2010 Mathematics Subject Classification:} 	32V20 (15A15 05A15 32H35)}

\thanks{The third author would like to acknowledge support from the Second Year Transformational Experience Program, the Honors and Scholars Enrichment Grant, and the Undergraduate Student Government's Academic Enrichment Grant at the Ohio State University. }

\date{\today}

\begin{abstract}  
We analyze a canonical construction of group-invariant CR Mappings between hyperquadrics due to D'Angelo. Given source hyperquadric of $Q(1,1)$, we determine the signature of the target hyperquadric for all finite subgroups of $SU(1,1)$. We also extend combinatorial results proven by Loehr, Warrington, and Wilf on determinants of sparse circulant determinants. We apply these results to study CR mappings invariant under finite subgroups of $U(1,1)$.

\end{abstract}

\maketitle

\section{Introduction}
The goal of this paper is to examine properties of group-invariant Cauchy-Riemann (CR) mappings, especially as relating to unitary groups with indefinite metric. A fundamental problem in CR geometry and higher dimensional complex analysis is to determine whether given two real hypersurfaces, $M$ and $N$ in $\C^n$ and $\C^N$ respectively, there exists a non-constant (smooth) CR mapping $f: M \to N$ (see for instance \cite{D3,BER,F1}). In general there are no such mappings; however for spheres and hyperquardrics (defined below) many such inequivalent mappings can exist. In recent years progress on this problem has led to the emerging field of {\it CR complexity theory} (see for instance the recent work of D'Angelo and Xiao in \cite{DX1} and the references therein). 

A natural starting point for this problem is to study mappings between spheres and hyperquadrics. Let $S^{2m-1}$ be the unit sphere in $\mathbb{C}^m$, and define the hyperquadric with signature pair $(a,b)$ as
$$Q(a,b)= \{z \in \mathbb{C}^{a+b}: \sum_{i=1}^{a}|z_i|^2-\sum_{j=a+1}^{a+b}|z_j|^2=1\}.
$$ Note $Q(a,0)=S^{2a-1}$. Since there are many inequivalent mappings $f: Q(a,b) \to Q(A,B)$ (for appropriate choices of $a,b,A,B$), it is reasonable to ask for the $f$ to satisfy additional constraints.  In the present paper we focus on the problem of constructing group-invariant mappings between hyperquadrics.

CR mappings invariant under finite group actions yield many interesting connections with other areas of mathematics, including number theory, combinatorics, algebraic geometry, and representation theory (see for instance \cite{LWW, D2,D3, DX1,DX2, G1,G2} and the references therein).
In this case, let $\Gamma$ be a finite subgroup of an indefinite unitary group $SU(a,b)$ or $U(a,b)$ (defined precisely in Section \ref{s:su}). We are interested in CR mappings satisfying 
$$\label{1}
f: Q(a,b) \rightarrow Q(A,B) \qquad \text{ and } \qquad f \circ \gamma = f \text{ for all } \gamma \in \Gamma.
$$ 

The problem of constructing group-invariant CR mappings has attracted substantial interest over the years (see \cite{F1,F2,F3,D1,D2,D3, G1, G2, GLV,DX1, DX2} and the references therein). D'Angelo and Lichtblau \cite{DL} gave a canonical construction of invariant-polynomial CR mappings, and they used this construction to study the CR spherical space form problem. The first author \cite{G1,G2} found the target hyperquadric for this canonical construction for all subgroups of $SU(2)$. In the present paper we build on this work in the case of the source hyperquadric $Q(1,1)$. 

The fundamental approach in this paper is to use the canonical construction of invariant CR mappings given  in \cite{DL}. Let $\Gamma$ be a finite subgroup of the indefinite unitary group $U(a,b)$, and define $$\langle z, w \rangle_b=\sum_{j=1}^a{z_j \overline{w_j}}-\sum_{j=a+1}^{a+b}{z_j \overline{w_j}}.$$ Following the construction in \cite{DL} of an invariant, Hermitian-symmetric polynomial, we introduce
$$
\Phi_{\Gamma}(z,\bar{z})= 1- \prod_{\gamma \in \Gamma}(1-\langle \gamma z,z \rangle_b).
$$

Expanding the product and diagonalizing the underlying matrix of coefficients gives
$$
\Phi_{\Gamma}(z,\bar{z})= ||F(z)||^2-||G(z)||^2
$$
where $F$ and $G$ are $\Gamma$-invariant, linearly-independent, holomorphic polynomials. Let $N^{+}(\Gamma)$ and $N^{-}(\Gamma)$ denote the numbers of components of $F$ and $G$ respectively. Therefore $\Phi_{\Gamma}$ generates an associated $\Gamma$-invariant CR-mapping $\phi_{\Gamma}$ given by $$\phi_{\Gamma}=F\oplus G : Q(a,b) \to Q(N^{+}(\Gamma),N^{-}(\Gamma)).$$ We define $S(\Gamma)=(N^{+}(\Gamma),N^{-}(\Gamma))$ to be the signature pair of $\Phi_{\Gamma}$ or the target hyperquadric of the associated CR mapping $\phi_\Gamma$. 
In the general case of an arbitrary finite group $\Gamma$, the resulting matrix of coefficients of $\Phi_{\Gamma}$ is often very large and difficult to diagonalize explicitly; however, some groups give a sparse or diagonal matrix, allowing interesting results to be proven. For instance, in \cite{G1} the first author shows that the target hyperquadric associated to the binary icosahedral group of order 120 has 40 positive eigenvalues and 22 negative eigenvalues.

Even in the case of cyclic subgroups, the combinatorics is interesting and difficult. Let $\Gamma_{p;q_1,q_2}$ be the cyclic subgroup of order $p$ in $U(1,1)$ generated by $$s=\begin{pmatrix} \omega^{q_1} & 0 \\ 0 & \omega^{q_2} \end{pmatrix}$$ where $\omega$ is a $p$-th primitive root of unity and $p$, $q_1$, and $q_2$ have no common factor. We assume $q_1$ and $q_2$ are minimally chosen and $0 \leq q_1 \leq q_2 < p$. The first main result gives the exact signature pair for all finite subgroups of $SU(1,1)$.

\begin{Thm}\label{thm:SU(1,1)} If $\Gamma$ is a subgroup of order $p$ in $SU(1,1)$, then
the signature pair is given by
$$
S(\Gamma)= \left(2, \frac{p}{2}\right) \text{ if }p\text{  is even}, \quad  \left(1, \frac{p+1}{2}\right) \text{ if }p\text{  is odd.}
$$

\end{Thm}

For a general cyclic subgroup $\Gamma_{p,q_1,q_2} \subset U(1,1)$, a more delicate combinatorial analysis is required. Loehr, Warrington, Wilf \cite{LWW} used circulant determinants in order to study the case 
$\Gamma_{p;1,q_2} \subset U(2)$. In \cite{LWW}, they posed the problem of generalizing their results to the general case $\Gamma_{p;q_1,q_2}$. We adapt and extend their arguments to complete the story of $\Gamma_{p;q_1,q_2}$ in Theorem \ref{thm:coefsign}. The underlying strategy follows the ideas of Loehr, Warrington, and Wilf in \cite{LWW}.

\begin{Thm}\label{thm:coefsign}
Suppose $\gcd(p,q_1,q_2) = 1$. In the polynomial
\[\Phi_{p,q_1,q_2}(z,\overline{z}) = 1 - \prod_{j = 1}^{p}(1 - |z_1|^2\omega^{q_1 j} - |z_2|^2\omega^{q_1 j}),\]
the monomials $C_{rs} |z_1|^{2r}|z_2|^{2s}$ which appear are exactly those for which $p|(rq_1 + sq_2)$, and the coefficients $C_{rs}$ of these monomials are positive if and only if $\gcd \left(q_1,q_2,\frac{rq_1 + sq_2}{p}\right)$
 is odd. 
\end{Thm}

This result can then be applied to estimate the signature pairs for the subgroups of $\Gamma_{p;q_1,q_2} \subset U(1,1)$. In this case we study the positivity ratio $\frac{N^+}{N}$ where $N=N^{+}+N^{-}$. Theorem \ref{thm:asymp} gives the asymptotic behavior of the positivity ratio for the source hyperquadric $Q(1,1)$, and Theorem \ref{thm:sphere} gives the behavior in the case of the source sphere $Q(2,0)=S^{3}$.

\begin{Thm}\label{thm:asymp} 
For $\Gamma_{p;q_1,q_2}\subset U(1,1)$ the asymptotic positivity ratio is given by
$$\lim_{p \to \infty} \frac{N^+(\Gamma_{p;q_1,q_2})}{N(\Gamma_{p;q_1,q_2})}=\begin{cases}
\frac{q_1 q_2+1}{4q_1 q_2} & \text{ for $q_1$ odd and $q_2$ odd.} \\
\frac{q_1(q_2-q_1)+1}{4q_1(q_2-q_1)} & \text{ for $q_1$ odd and $q_2$ even,}

\end{cases}$$

Furthermore, 
$$\lim_{p \to \infty} \frac{N^+(\Gamma_{2p-1;q_1,q_2})}{N(\Gamma_{2p-1;q_1,q_2})}=
\frac{q_2 (q_2-q_1)+1}{4q_2 (q_2-q_1)} \text{ for $q_1$ even and $q_2$ odd.}$$ and
$$\lim_{p \to \infty} \frac{N^+(\Gamma_{2p;q_1,q_2})}{N(\Gamma_{2p;q_1,q_2})}=
\frac{3q_2(q_2-q_1)-1}{4q_2(q_2-q_1)} \text{ for $q_1$ even and $q_2$ odd,}$$
\end{Thm}

\begin{Thm}\label{thm:sphere} 
For $\Gamma_{p;q_1,q_2}\subset U(2)$ the asymptotic positivity ratio is given by
$$\lim_{p \to \infty} \frac{N^+(\Gamma_{p;q_1,q_2})}{N(\Gamma_{p;q_1,q_2})}=\begin{cases}\frac{3q_1(q_2-q_1)+1}{4q_1(q_2-q_1)} & \text{ for $q_1$ odd and $q_2$ even,}\\
\frac{3q_1 q_2+1}{4q_1 q_2} & \text{ for $q_1$ odd and $q_2$ odd,}\\
\frac{3q_2(q_2-q_1)-1}{4q_2(q_2-q_1)} & \text{ for $q_1$ even and $q_2$ odd,} \end{cases}$$
and hence
$$\lim_{q_1\to \infty} \lim_{q_2\to \infty} \lim_{p\to \infty}{\frac{N^{+}(\Gamma_{p;q_1,q_2})}{N(\Gamma_{p;q_1,q_2})}}=\frac{3}{4}.$$
\end{Thm}

The paper is organized as follows. In Section \ref{s:su} we use D'Angelo's construction to prove Theorem \ref{thm:SU(1,1)}. In Section \ref{s:comb} we will extend the results of \cite{LWW} to prove Theorem \ref{thm:coefsign}. This generalization will allow us, in Section \ref{s:general}, to determine signature pairs for a general finite subgroup of $U(1,1)$ and prove Theorems \ref{thm:asymp} and \ref{thm:sphere}.

\section{Results from $SU(1,1)$} \label{s:su}

In this section we recall some background material and prove Theorem \ref{thm:SU(1,1)}. We begin by giving precise definitions for the indefinite unitary groups $U(1,1)$ and $SU(1,1)$; namely,
$$
U(1,1)=\{A \in GL(2, \mathbb{C}): A^* J A = J\} \quad  \text{ where } \quad J= \begin{bmatrix}
1 && 0\\
0 && -1\\
\end{bmatrix}, \text{ and }
$$
$$SU(1,1)=\{A \in U(1,1): \det A = 1\}.$$

Let $\Gamma$ be a finite subgroup of the indefinite unitary group $SU(1,1)$. Simon proves in \cite[Thm 10.4.15]{S1} that every compact subgroup of $U(1,1)$ is abelian. This theorem can be used to determine the finite subgroups $\Gamma$ in $U(1,1)$. Since the only simple abelian groups are the cyclic groups of prime order, the finite subgroups of $U(1,1)$ must be these groups and their direct products. In the case of $SU(1,1)$ we can further refine our restrictions by noting that the group is isomorphic to $SL(2,\mathbb{R})$ as proven by Simon in \cite[Prop 10.4.1]{S1}. It is well known that the only finite subgroups of $SL(2,\mathbb{R})$ are cyclic. Thus, we can restrict our attention to the problem of cyclic groups in $SU(1,1)$ for the rest of this section. First consider the case of a cyclic group of order $p$. Recall that in $U(1,1)$, these groups can be considered as generated by the matrix
$$
s= \begin{bmatrix}
\omega^{q_1} && 0\\
0 && \omega^{q_2}\\
\end{bmatrix}.
$$
Here $\omega$ is a primitive $p$-th root of unity and $q_1,q_2,p$ have no common factor. Without loss of generality we assume that $q_1$ and $q_2$ are chosen minimally. For instance the parameters $(p;q_1,q_2)=(7;1,2)$ generates the same subgroup as $(p;q_1,q_2)=(7;2,4)$. Thus we assume $q_1$ and $q_2$ have no common factor. Since the determinant of $s$ must equal one in the $SU(1,1)$ case, $q_1+q_2$ must equal $p$. We can then simplify to the generator $s$ below without loss of generality:
$$
s= \begin{bmatrix}
\omega && 0\\
0 && \omega^{p-1}\\
\end{bmatrix}.
$$
Using the diagonal form of the generator $s$ allows us to simplify the invariant Hermitian polynomial $\Phi_{\Gamma}$ as follows:
$$
\Phi_{\Gamma}(z,\bar{z}):=\Phi_{p;1,p-1}(z,\bar{z})= 1- \prod_{\gamma \in \Gamma}(1-\langle \gamma z,z \rangle_1)=1-\prod_{j=1}^{p}(1-w^j |z_1|^2 + w^{j(p-1)}|z_2|^2)=\sum_{\alpha, \beta} C_{\alpha \beta} z^{\alpha}\bar{z}^{\beta} 
$$
where $C_{\alpha\beta}$ is the matrix of coefficients. We illustrate this approach with an example. Let $p=2$, then
$$
\Phi_{2;1,1}(z,\bar{z})= 1- \prod_{\gamma\in \Gamma}(1-\langle \gamma z,z \rangle)=\bar{z_1}^2 z_1^2 - 2 \bar{z_1} \bar{z_2} z_1 z_2 + \bar{z_2}^2 z_2^2.
$$
In matrix form this expands to:
$$
\Phi_{2;1,1}(z,\bar{z})=\begin{bmatrix}
 \bar{z_1}^2 \\ \bar{z_1} \bar{z_2} \\ \bar{z_2}^2
\end{bmatrix}^{T} 
\begin{bmatrix}
 1 & 0 & 0  \\ 
 0 & -2 & 0 \\ 
 0 & 0 & 1  \\ 
\end{bmatrix}
\begin{bmatrix}
 z_1^2 \\  z_1 z_2 \\  z_2^2
\end{bmatrix}.
$$
Since the eigenvalues of this matrix are $1,1, \text{ and } -2$, the resulting signature pair for the cyclic group of order 2 in $SU(1,1)$ is $(2,1)$. The associated CR mapping is given by
$$
\phi: Q(1,1) \mapsto Q(2,1), \quad \phi(z_1,z_2)=(z_1^2,z_2^2,\sqrt{2}z_1 z_2).
$$
We move on to the general case of an arbitrary cyclic subgroup of $SU(1,1)$. It is here useful to adapt results for cyclic subgroups of $SU(2)$ from \cite{D1, G1, G2}. D'Angelo \cite{D1} gave an explicit expression for $\Phi_{p;1,p-1}$ for cyclic subgroups of $SU(2)$; namely, 
$$
\Phi_{p;1,p-1}(z,\bar{z})=1-\prod_{j=1}^{p}(1-\omega^j |z_1|^2 - \omega^{j(p-1)}|z_2|^2)=|z_1|^{2p}+|z_2|^{2p}+\sum_{j=1}^{\lfloor \frac{p}{2} \rfloor}(-1)^{j-1}k_j|z_1|^{2j}|z_2|^{2j}
$$
for some positive (explicit) coefficients $k_j$. We can modify this result to the $SU(1,1)$ case by changing the necessary signs in the expression to get
$$
1-\prod_{j=1}^{p}(1-w^j |z_1|^2 + w^{j(p-1)}|z_2|^2)=|z_1|^{2p}+(-1)^p|z_2|^{2p}+\sum_{j=1}^{\lfloor \frac{p}{2} \rfloor}(-1)^j(-1)^{j-1}k_j|z_1|^{2j}|z_2|^{2j}. 
$$
Simplifying the powers of $-1$ then yields $\Phi_{p;1,p-1}$ in $SU(1,1)$ in the following form
$$
\Phi_{p;1,p-1}(z,\bar{z})=|z_1|^{2p}+(-1)^p|z_2|^{2p}-\sum_{j=1}^{\lfloor \frac{p}{2} \rfloor}k_j|z_1|^{2j}|z_2|^{2j}.
$$
From this final expression, we determine the signature pairs for all cyclic subgroups of $SU(1,1)$:
$$S(\Gamma)= \left(2, \frac{p}{2}\right) \text{ for }p\text{ even}, \quad  \left(1, \frac{p+1}{2}\right) \text{ for }p\text{ odd.}$$
This result is summarized in Theorem \ref{thm:SU(1,1)}.

\section{Proof of Theorem 2} \label{s:comb}
  
In this section, we recall and extend a combinatorial approach to $\Phi_{\Gamma}$ due to Loehr, Warrington, and Wilf \cite{LWW}. In particular, they gave a combinatorial method for determining information about the coefficients of the invariant polynomial $\Phi_{\Gamma}$ for cyclic subgroups $\Gamma$ of $U(2)$. Replacing $|z_1|^2$ and $|z_2|^2$ with $x$ and $y$ respectively, we introduce the following notation for $\Phi_{\Gamma_{p;q_1,q_2}}$:
\begin{equation}\label{eq:1-1}
\Phi_{p;q_1,q_2}(x,y) = 1- \prod_{j = 0}^{p - 1}(1 - x\omega^{q_1 j} - y\omega^{q_2 j}).
\end{equation}
In particular, Loehr, Warrington, and Wilf determined the signs of the coefficients of $1-\Phi_{p;q_1,q_2}(x,y)$ in the cases where either of the pairs $(p,q_1)$ or $(p,q_2)$ are coprime. They posed the general problem of determining the signs for arbitrary $p, q_1, q_2$ in \cite{LWW}. We give a complete answer to their question in this section by adapting and extending their techniques. We prove the set to which their results apply is the following
\[\{(p,q_1,q_2) : \gcd(p,q_1,q_2)= 1\}; \] this result is Theorem \ref{thm:coefsign}.

This section relies on two insights from \cite{LWW} in order to provide information about the coefficients of $\Phi_{p;q_1,q_2}(x,y)$. Let $a_{p;q_1,q_2}(r,s)$ denote the coefficient of $x^ry^s$ in $ \Phi_{p;q_1,q_2}(x,y)$. We define $l = \frac{rq_1 + sq_2}{p}$ as the weight of the monomial $a_{p;q_1,q_2}(r,s)x^ry^s$.   The first is to express the product expansion \eqref{eq:1-1} as the determinant of a 3-line circulant matrix. The second is to compute this determinant by counting permutations with certain fixed points.

We pause to consider the smallest case not covered by \cite{LWW}. For example,
\[\Phi_{6;2,3}(x,y) =  2 x^3 - x^6 + 3 y^2 + 6 x^3 y^2 - 3 y^4 + y^6.
 \] In this case, we have $(p,q_1,q_2) = (6,2,3)$ in which $\gcd(6,2) = 2 > 1$ and $\gcd(6,3) = 3 > 1$, thus neither of the pairs are coprime, but $\gcd(6,2,3) = 1$. The weight one terms are $2x^3$ and $3y^2$; the weight two terms are $-x^6$, $6x^3y^2$, and $-3y^4$; the weight three term is $y^6$. Notice that the odd weight terms are all positive, and the even weight terms alternate in sign, as predicted by the Loehr, Warrington, Wilf condition (as stated in Theorem 2). 

\subsection{Setup}
In order to set up the combinatorial argument used in the proof of Theorem \ref{thm:coefsign}, we first express the polynomial $\Phi_{p;q_1,q_2}$ as the determinant of a $3$-line $p \times p$ square circulant matrix. Recall an $n\times n$ circulant matrix is of the form $$\text{circ}(c_1,\dots,c_n)=
\begin{pmatrix}
c_1 & c_2 & \dots & c_{n-1} & c_n\\
c_n & c_1 & \dots & c_{n-2} & c_{n-1}\\
\vdots & & \ddots & &\\
c_3 & c_4 & \dots & c_1 & c_2\\
c_2 & c_3 & \dots & c_n & c_1\\
\end{pmatrix}.$$ A well-known property of circulant matrices is the following formula for the determinant  (see \cite{RM} for a proof of this formula and other properties of circulant matrices):
$$\det(\text{circ}(c_1,\dots,c_n))=\prod_{j = 1}^{n}\left(c_1 +c_2 \omega^j +\dots +c_n \omega^{(n-1)j}\right)$$ where $\omega$ is primitive $n$-th root of unity. Finally define $$d_j=\begin{cases}1 & j=1,\\
-x & j=q_1+1,\\
-y & j=q_2+1,\\
0 & \text{otherwise.}
\end{cases}$$
Thus we can express $\Phi_{p;q_1,q_2}$ as a circulant determinant:
$$\Phi_{p;q_1,q_2}(x,y)=1-\det(C_{p;q_1,q_2})$$ where $C_{p;q_1,q_2}=\text{circ}(d_1,\dots,d_p)$.
 
 \subsection{Computing the Determinant of $C_{p;q_1,q_2}$}
 
We expand the determinant of the $n \times n$ matrix $C=[a_{ij}]$:
\begin{equation}\label{eq:1-3}
det(C) = \sum_{\sigma \in S_n} sign(\sigma)a_{1,\sigma(1)}a_{2,\sigma(2)} \cdots a_{n,\sigma(n)}.
\end{equation}
We illustrate the case where $(p,q_1,q_2) = (6,2,3)$. Here we have
\[\det(C_{6;2,3})=\left[
\begin{array}{cccccc}
1 & 0 & -x & -y & 0 & 0\\
0 & 1 & 0 & -x & -y & 0\\
0 & 0 & 1 & 0 & -x & -y\\
-y & 0 & 0 & 1 & 0 & -x\\
-x & -y & 0 & 0 & 1 & 0\\
0 & -x & -y & 0 & 0 & 1\\
\end{array} \right]. \]

Following \cite{LWW}, we introduce the set of permutations $T_{p;q_1,q_2}(r,s)$ contributing to the monomial $x^r y^s$ in $\det(C_{p;q_1,q_2})$. We reproduce the definition below from \cite{LWW}.

We define the $T_{p;q_1,q_2}(r,s)$ as the set of all permutations $\sigma$ with
\begin{itemize}
\item $p - r - s$ fixed points ($0$-steps);
\item $r$ values of $j$ where $\sigma(j)-j$ is congruent to $q_1$ mod $p$ ($q_1$-steps);
\item $s$ values of $j$ where $\sigma(j)-j$ is congruent to $q_2$ mod $p$ ($q_2$-steps).
\end{itemize}

The set $T_{6;2,3}(3,2)$, for example, contains the following permutations: 
$$
(2 4 6 3 5),
(1 3 5 2 4),
(1 3 6 2 4),
(1 3 6 2 5),
(1 4 6 3 5),
(1 4 6 2 5),
$$
and corresponds to the coefficient of the monomial $6 x^3 y^2$ in the polynomial $\Phi_{6;2,3}(x,y)$. 

We show that each permutation in $T_{p;q_1,q_2}(r,s)$  has the same cycle structure when we restrict ourselves to the condition that $\gcd (p,q_1,q_2) = 1$. It is essential to the argument that we show identical cycle structure within $T_{p;q_1,q_2}(r,s)$, since this allows us to say that all permutations in $T_{p;q_1,q_2}(r,s)$ have the same sign, which gives us
\[|T_{p;q_1,q_2}(r,s)| = |a_{p;q_1,q_2}(r,s)|, \]
as well as a method for determining the signs of the coefficients. 

Furthermore, in everything that follows, we restrict ourselves to the case where both of the pairs $(p,q_1)$ and $(p,q_2)$ have common factors, since when either of these pairs are coprime, we have a situation equivalent to one in which $q_1 = 1$, which is covered in \cite{LWW}. Since we assume $q_1$ and $q_2$ are chosen minimally, we also assume $q_1$ and $q_2$ are coprime.

\subsection{Cycle structure}

We proceed similarly to \cite{LWW} and adapt some of their notation, making slight modifications to allow for an arbitrary value for $q_1$. We note here that Lemmas \ref{lem:weight}, \ref{lem:risili}, and \ref{lem:idcycle} are straightforward extensions of Lemmas 7, 8, and 12 of \cite{LWW}. Lemmas \ref{lem:weight} and \ref{lem:risili} are only stated here, since their proofs follow simply by replacing appropriate values of $1$ with $q_1$ in the proofs from \cite{LWW}. 

Let $\sigma \in T_{p;q_1,q_2}(r,s)$. We decompose $\sigma$ into disjoint cycles of length greater than 1 to exclude fixed points: 
\[\sigma = C_1C_2 \cdots C_k. \]
We let $C_i$ be represented by $(x_i; w_i)$, where $x_i$ is an arbitrary point in our cycle, the ``starting point", and $w_i$ is a word (in this case an $r_i + s_i$-tuple) which specifies each of the ``steps", in order, within $C_i$. Here, we define $r_i$ and $s_i$ as the number of $q_1$-steps and $q_2$-steps in $C_i$. We define $w_i(x) = w_{i_x}$, i.e. the $x$-th component of $w_i$. For example if $w_i = (2,3,2,2,3)$, then $w_i(5) = 3, w_i(3) = 2$, etc. So we have: 
\[w_i = (w_i(1),w_i(2),...,w_i(r_i+ s_i)). \]
We define 
\[\sigma^t(x_i) \equiv x_i + \sum_{j = 1}^tw_i(j) \Mod p\]
where $t$ is a non-negative integer. As in \cite{LWW} we take our mod $p$ statement to indicate a restriction to values from $1$ to $p$ as opposed to $0$ to $p - 1$. 

\begin{Lem}\label{lem:weight}
If $T_{p;q_1,q_2}(r,s)$ is nonempty, then $p|(rq_1 + sq_2)$. 
\end{Lem}
\begin{Lem}\label{lem:risili}
If $T_{p;q_1,q_2}(r,s)$ is nonempty, then $\gcd (r_i,s_i,l_i) = 1$ for $1 \leq i \leq k$. 
\end{Lem}

Recall that we defined our modulo operation on the set $[p] = \{1,2,...,p\}$. Let $m \in \mathbb{Z}$ such that $1\leq m \leq p-1$, and define $j = \frac{\lcm (m,p)}{m}$. Now let $x_1, x_2, ..., x_n \in [p]$. 
\begin{Def}

We say the sequence $(x_1,x_2,...,x_n)$ is m-ordered on $[p]$ if: 
\begin{itemize}
\item $3 \leq n \leq j$
\item $x_i \equiv x_j \Mod{\gcd (p,m)} \forall i,j$
\item In the clockwise traversal of [p] via $m$-steps, starting with $x_1$, we hit $x_i$ before $x_j$ if and only if $i < j$. 
\end{itemize}
\end{Def}
We let $3 \leq n \leq j$ because our definition is not meaningful with less than three points (any two points which satisfy the second condition would be $m$-ordered). The second condition is due to the fact that our traversal by $m$-steps will not hit any element of $[p]$ which is not in the same equivalence class as $x_1$ modulo $\gcd(p,m)$. We also note here that our definition of $m$-ordered is invariant under rotations about the set $[p]$. \\

\begin{Lem}\label{lem:riequal}
If $\sigma \in T_{p;q_1,q_2}(r,s)$, and $\gcd (p,q_1,q_2) = 1$,  we must have $r_1 = r_2 = \cdots = r_k$ and $s_1 = s_2 = \cdots = s_k$. 
\end{Lem}

\begin{proof}
The proof follows the structure of the corresponding proof in \cite{LWW}. We simply adapt their notation and approach to accommodate arbitrary values of $q_1$. 

We let $C_k,C_l$ be two distinct, nontrivial cycles in $T_{p;q_1,q_2}(r,s)$. We write $C_k = (x_k,w_k)$ and group the $q_1$-steps in $w_k$ preceding each $q_2$-step such that:
 \[w_k = (q_1^{\rho_1},q_2,...,q_1^{\rho_{s_k}},q_2)\]
 where $\displaystyle \sum_{i}\rho_i = r_k$. Note that use of exponents in $w_k$ represents a series of ``steps" in a row within the word. We wish to show that $s_l = s_k$. 
 
We first argue that $s_l \geq s_k$. If $s_k = 0$, there is nothing to prove, so we assume $s_k \geq 1$, that is we assume $C_k$ contains at least one $q_2$-step. We will proceed by partitioning the points in $[p]$ fixed by $C_k$ into sets which we will define in terms of the $q_2$-steps, $\{e_i\}$, and the images of the $q_2$-steps, $\{d_i\}$. 

So we set $d_1 = x_k$, and $e_1 = C_k^{\rho_1}(x)$. If $s_k > 1$, further define $d_i = C_k(e_{i - 1})$, $e_i = C_k^{\rho_i}(d_i)$, for $1 \leq i \leq r_k$. 

We assert that there exists a unique permutation $U$ such that for all $x$ in the set \[\{x \in [p]:(e_j,x,d_{U(j)})\text{ is }q_1\text{-ordered}\},\] $x$ is not in $\{d_j\}$. In other words, we assert that for each $q_2$-step ($e$ point), we have a $d$ point which is ``hit" first when we traverse $[p]$ in $q_1$-steps starting from the given $e$ point. Our assertion could only be false if there exists an $x \in \{e_i\}$ such that $x \not\equiv d_i \Mod{\gcd(p,q_1)}$ for each $i$. But we know by construction of $\{e_i\}$ that $x = C_k^{\rho_j}(d_j)$ for some $j$ which implies that $x \equiv d_j \Mod{\gcd(p,q_1)},$ since $x$ is obtained from $d_j$ entirely by $q_1$-steps within $C_k$, so the assertion holds. 

Now we define
 \begin{equation}\label{eq:1-4}
 V_j = \{x \in [p]: (e_j,x,d_{U(j)}) \text{ is }q_1\text{-ordered}\}.
 \end{equation}
 These are the points in $[p]$ not hit by $C_k$ when we move via $q_1$-steps starting with $e_j$, stopping before the first $d$ point encountered ($d_{U(j)}$). We know all points in $V_j$ must be fixed by $C_k$ by definition of $U$.
 
 We also define functions $\pi, f, g$ as:
 \[\pi(x) = x + q_2 \Mod p,\]
 \[f(x) = x - q_1 \Mod p,\]
 \[g(x) = x + q_1 \Mod p, \]
 \\\\
 and the set $W_j$ as:
 \begin{equation}\label{eq:1-5}
 \begin{aligned}
W_j &= \{y \in [p] :  \left(f(\pi(e_j)),y,g(e_{j + 1})\right) \text{ is } q_1\text{-ordered}\}\\
 &=\{y \in [p] : \left(f(d_{j + 1}),y,g(e_{j + 1})\right) \text{ is } q_1\text{-ordered}\}.
  \end{aligned}
  \end{equation}
  We use the functions $f,g$ to specify the endpoints of the interval $[d_{j + 1},e_{j + 1}] \subset [p]$. Note that $C_k(x) \neq x$ holds for each $x \in W_j,$ and for every $j$ by definition of $e_j$.

We claim that for each  $x \in V_j$, if $\pi(x) \in C_k$, $\pi(x) \in W_j$. So let $x \in V_j$ such that $\pi(x) \in C_k$. Note that all points $x$ in $V_j$ must be such that $(e_j,x,d_{U(j)})$ is $q_1$-ordered. Then we know that $(\pi(e_j),\pi(x),\pi(d_{U(j)})) = (d_{j + 1},\pi(x),\pi(d_{U(j)}))$ must also be $q_1$-ordered since the $q_1$-ordered condition is invariant under a rotation about $[p]$. Note
 \[W_j \cup V_{j + 1}  = \{x \in [p]: (f(d_{j + 1}),x,d_{U(j + 1)}) \text{ is }q_1\text{-ordered}\},\]
which means $ \pi(x)$ is in $ W_j \cup V_{j + 1}$ which in turn implies that $\pi(x)$ is in  $W_j$ since all points in $V_{j + 1}$ are fixed by $C_k$. 
 
 We also claim that for each $x$ such that $C(x) = x$, there exists a $V_j$ such that $x \in V_j$. \\

 For each $ x \in [\gcd(q_1,p)]$, there exists an  $e_j$ such that $e_j \equiv x \Mod{\gcd(q_1,p)}$, since we know we have at least a single $q_2$-step, and to complete the cycle, we must return to the same equivalence class modulo $\gcd(q_1,p)$. Thus, since $q_1,q_2$ coprime, we know there exists an $e_j$ for each element of $[\gcd(q_1,p)]$. \\
 
Now let $C_k(x) = x$. Then because there exists an  $e_j$ such that $e_j \equiv x \Mod{\gcd(q_1,p)}$, we know $x$ is in $\displaystyle \bigcup_j V_j$ since by construction of the $V_j$ sets, $\displaystyle \bigcup_j V_j$ partitions all points congruent to some $e_j$ modulo $\gcd(q_1,p)$. 

Also since the $V_j$ sets do not overlap, we know $x$ lies within in a unique $V_j$. So let $C_l(x) \neq x$ which implies $C_k(x) = x $ which gives us that $x$ is in $V_j$ for a unique $j$. Now if $x$ is a $q_1$-step of $C_l$, $C_l(x)$ is in $V_j$ as well, since $q_1 \neq q_2$ and $C_k,C_l$ disjoint. If $x$ is a $q_2$-step of $C_l$, we know $(e_j,x,d_{U(j)})$ is $q_1$-ordered which tells us $(\pi(e_j),\pi(x),\pi(d_{U(j)})) = (d_{j + 1},\pi(x),\pi(d_{U(j)}))$ is also $q_1$-ordered. Then as noted above, $(d_{j + 1},\pi(x),\pi(d_{U(j)}))$ being  $q_1$-ordered tells us that $\pi(x)$ is in $W_j \cup V_{j + 1}$, which implies that $\pi(x)$ is in $V_{j + 1}$. Iterating, we see that $x$ orbits through all $V_j$'s which implies  $s_l \geq s_k$. Arguing with the roles of $C_k,C_l$ switched, $s_k = s_l$. 

Now we also know that $r_kq_1 + s_k q_2 = \lambda_k p$ and $r_lq_1 + s_k q_2 = \lambda_l p$, since we assumed $\sigma$ was in $T_{p;q_1,q_2}(r,s)$. Then $r_k - r_l = (\lambda_k - \lambda_l)p$. Since $s_k = s_l > 0$, we know $0 \leq r_k,r_l < p$, thus
\[-p < r_k - r_l < p.\]
Since we know $p \mid (r_k - r_l)$, we conclude that $r_k= r_l$. 
\end{proof}

\begin{Ex}
\end{Ex}
We consider a permutation $\sigma \in T_{24;3,16}(16,6)$. 
We have $\sigma = C_1C_2$, where
\begin{equation}\label{eq:ex1}
\begin{aligned}
C_1 =& (20,23,2,18,21,24,3,19,22,1,4)\\
C_2 =& (7,10,13,5,8,11,14,6,9,12,15),
\end{aligned}
\end{equation}
and the two fixed points of $\sigma$ are $16$ and $17$. We write $C_1$ in the form $(x_i;w_i)$
\[C_1 = (20;q_1^2q_2q_1^3q_2q_1^3q_2) = (20;3^2 \cdot16 \cdot 3^3 \cdot16 \cdot 3^3\cdot 16),\]
and observe that $r_1 = 8$ and $s_1 = 3$. We also have from definitions of $d_i,e_i$ in Lemma \ref{lem:riequal} that 
\begin{equation}\label{eq:ex2}
\begin{aligned}
(d_1, d_2, d_3) =& (20,18,19)\\
(e_1, e_2, e_3) =& (2,3,4). 
\end{aligned}
\end{equation}
Also, the permutation $U$ defined in Lemma \ref{lem:riequal}  written in one-line notation, happens to be the identity permutation, since $d_{U(1)} = 20,d_{U(2)} = 18, d_{U(3)} = 19$. We now write out the $V_j$ sets
\begin{equation}\label{eq:ex3}
\begin{aligned}
V_1 =& \{5,8,11,14,17\}\\
V_2 =& \{6,9,12,15\}\\
V_3 =& \{7,10,13,16\},\\
\end{aligned}
\end{equation}
observing that they partition all elements of $[24]$ not hit by $C_1$. We also write out the $W_j$ sets
\begin{equation}\label{eq:ex3}
\begin{aligned}
W_1 =& \{18,21,24,3\}\\
W_2 =& \{19,22,1,4\}\\
W_3 =& \{20,23,2\}. \\
\end{aligned}
\end{equation}
Note that for any $x$ in some $V_j$, if $x$ is a $16$-step in some cycle $C_i$, $i\neq 1$, then we must have that $C_i(x)$ is in $V_k$ for $k \neq j$. Take $x = 15$, for example, and suppose $15$ is a $16$-step of $C_i$. Then $C_i(15) = 7$, and $7$ is in $V_3$. Indeed this is the case in $C_2$. 

\begin{Lem}\label{lem:idcycle}
If $\sigma $ is in $T_{p;q_1,q_2}(r,s)$, and $\gcd (p,q_1,q_2) = 1$, we must have that $k = \gcd (r,s,l), r_i = r/k, s_i = s/k$, for all $i$, and all elements of $T_{p;q_1,q_2}(r,s)$ have identical cycle structure. Subsequently, $sgn(\sigma) = (-1)^{r + s + \gcd (r,s,l)}$. 
\end{Lem}
\begin{proof}
Recall that we defined $k$ to be the number of disjoint cycles in our permutation $\sigma = C_1 C_2 \cdots C_k$. Let $\gcd (p,q_1,q_2) = 1$ and let $\sigma$ be in $T_{p;q_1,q_2}(r,s)$. We know $\sum_{i = 1}^k r_i = r$ and $\sum_{i = 1}^k s_i = s$, so the result of Lemma \ref{lem:riequal} gives us that $r_i = r/k$ and $s_i = s/k$ for all $i$. Then for each i,
\[l_i = \frac{r_iq_1 + s_iq_2}{p} = \frac{\frac{rq_1 + sq_2}{p}}{k} = \frac{l}{k}.\]

Now Lemma \ref{lem:risili} gives us that $\gcd (r_i,s_i,l_i) = 1$. Then: 
\[k = k \gcd (r_i,s_i,l_i) = \gcd (kr_i,ks_i,kl_i) = \gcd (r,s,l).\]
Also note that the sign of $\sigma$ is the parity of the quantity $p - c$ where $c$ denotes the total number of cycles in $\sigma$, including $1$-cycles. Recall that we have $p - r -s$ $1$-cycles, which gives us
\[sgn(\sigma) = (-1)^{p - (k + (p - r - s))} = (-1)^{r + s + \gcd (r,s,l)}.\]
\end{proof}

We are now able to formulate the following combinatorial formula for the coefficients, which is an extension of a result from Section 2 of \cite{LWW}.
\begin{Lem}\label{lem:coeffs}
Given $\gcd(p,q_1,q_2) = 1$, if $a_{p;q_1,q_2}(r,s)$ is defined as the coefficient of $x^ry^s$ in $\Phi_{p;q_1,q_2}(x,y)$,
\[a_{p;q_1,q_2}(r,s) = (-1)^{\gcd(r,s,l)+1}|T_{p;q_1,q_2}(r,s)|.\]
\end{Lem}

\begin{proof}
Lemma \ref{lem:idcycle} gives us two new pieces of information about the coefficients $a_{p;q_1,q_2}(r,s)$. First is the following equality:
\begin{equation}\label{eq:1-6}
|T_{p;q_1,q_2}(r,s)| = |a_{p;q_1,q_2}(r,s)|.
\end{equation}
 Our determinant formula gave us that the absolute value of $a(r,s)$ is equal to the absolute value of the sum of the signs of the permutations in $T$. However, we now know that the signs of permutations in $T$ are uniquely determined by $p,q_1,q_2,r,s$, which implies all the signs of the permutations in $T$ are the same, hence the above equality \eqref{eq:1-6}. Furthermore we now have a formula for the signs of the elements of $T$, and so we can give a complete formula for the coefficients. Since $\Phi$ is defined as $1-\det(C_\Phi)$, we have the desired result
 
 \begin{equation}\label{eq:1-7}
 \begin{aligned}
 a_{p;q_1,q_2}(r,s) & = (-1) (-1)^{r + s}|T_{p;q_1,q_2}(r,s)|(-1)^{r + s + \gcd (r,s,l)}\\
 & = (-1)^{\gcd (r,s,l)+1}|T_{p;q_1,q_2}(r,s)|. 
 \end{aligned}
 \end{equation}
 Thus if the weight $l$ is odd, $\gcd (r,s,l)$ is odd, and hence $a_{p;q_1,q_2}(r,s)>0$. If the weight $l$ is even, $\gcd (r,s,l)$ alternates between even and odd, and hence $a_{p;q_1,q_2}(r,s)$ alternates sign.
 
 \end{proof}

\subsection{Construction of elements in $T_{p;q_1,q_2}(r,s)$}
In Lemma \ref{lem:weight} we see that $T_{p;q_1,q_2}(r,s)$ being nonempty implies $ p \mid (rq_1 + sq_2)$. Because of equation \eqref{eq:1-7} and the result of Lemma \ref{lem:weight}, we have
\begin{equation}\label{eq:1-8}
a_{p;q_1,q_2}(r,s) \neq 0 \Rightarrow p \mid (rq_1 + sq_2).
\end{equation}
However, we'd like to say precisely when our coefficients are nonzero. 
Thus, this section will serve to prove the converse of \eqref{eq:1-8} (note this is generally not true in the higher dimensional setting; see \cite{G2} for an example). The lemmas and theorems required to prove this follow in the same fashion as in \cite{LWW}, and we omit the proofs that follow immediately from their analogues in \cite{LWW}. We use their notation as much as possible to allow the reader to easily compare.

So we assume $rq_1 + sq_2 = lp$ and $gcd(r,s,l) = 1$. Recall that $k = gcd(r,s,l)$, so $k = 1$ means we are constructing a permutation $\sigma$ which should consist of only a single cycle. We define the lattice path
\begin{equation}\label{eq:1-9}
\nu = [\nu_0 = (0,0),\nu_1,\nu_2,...,\nu_{r + s} = (r,s)]
\end{equation}
where either $\nu_i - \nu_{i -1} = (1,0)$ or $\nu_i - \nu_{i -1} = (0,1)$ for $i < 0$.

Let $v$ be an $(r + s)$-tuple in $\{q_1,q_2\}^{r + s}$ with exactly $r$ $q_1$'s and $s$ $q_2$'s. Note that $\{q_1,q_2\}^{r + s}$ is an $r + s$-dimensional vector space which implies each of $v$'s components must be either a $q_1$ or a $q_2$. If $\nu_i - \nu_{i -1} = (1,0)$, we let the $i$-th entry in $v$ be a $q_1$. Else we have $\nu_i - \nu_{i -1} = (0,1)$, in which case we set the $i$-th entry to a $q_2$. We will show that $(x;v)$ is a well-defined element of $T_{p;q_1,q_2}(r,s)$ for all $x$ in $[p]$. 

We recursively define $\nu$ with $\nu_i = (x_i,y_i)$:
\begin{equation}\label{eq:1-10}
\nu_i = 
\begin{cases}
nu_i + (1,0),\text{ if }sx_i \leq ry_i,\\
nu_i + (0,1),\text{ if }sx_i \leq ry_i.\\

\end{cases}
\end{equation}

We omit the proof of the following lemma since it is exactly the same as it appears in \cite{LWW}. 
\begin{Lem}\label{lem:rs1}
Given $\nu$ constructed as above, let $0 \neq i, j \leq r + s$ and write $\nu_i = (x_i,y_i)$ and $\nu_j = (x_j,y_j)$. If $b = y_j - y_i$ and $a = x_j - x_i$, then $|(as - br)| \leq r + s - 1$. 
\end{Lem}

\begin{Lem}\label{lem:sarb}
If $\gcd(p,q_1,q_2) = 1$, and $a,b,r,s,p,q_1,q_2$ are integers such that $p \mid (aq_1 + bq_2)$ and $p \mid (rq_1 + sq_2)$, then $sa -rb = 0$ or $|sa - rb| \geq p$. 
\end{Lem}
\begin{proof} When $q_1$ or $q_2$ is 1, the result appears in \cite{LWW}, so we assume $q_1, q_2>1$.
Let $aq_1 + bq_2 = mp$ and $rq_1 + sq_2 = lp$. Then
\[|saq_1 -rbq_1| = |saq_1 + sbq_2 - sbq_2 -rbq_1| = |s(aq_1 + bq_2) - b(rq_1 + sq_2)| = |p(sm - bl)|,\]
which implies
\[|sa -rb| = \left|\frac{p(sm - bl)}{q_1} \right|.\]
We claim $q_1 \mid (sm - bl)$. Suppose for contradiction that $q_1 \nmid (sm - bl)$. Then 
\[sm - bl = q_1\gamma + \delta,\]
 where $\gamma, \delta$ are integers, and $0 < \delta < q_1$. 
 Observe 
\[q_2(sm-bl) = q_2(q_1 \gamma + \delta) = q_1(al-rm) \]
which implies
\[q_2 \delta = q_1(al-rm -q_2 \gamma).\]

But since $\gcd(q_1,q_2) = 1$ and $q_1 \neq 1$, we must have $q_1 \mid \delta$, which is a contradiction since we said that $0 < \delta < q_1$. Thus we must have that $q_1 \mid (sm-bl)$. This implies $ \frac{(sm - bl)}{q_1}$ is an integer. Since this integer is either $0$ or at least $1$, we have the desired result because 
\[|sa -rb| = \left|\frac{p(sm - bl)}{q_1} \right|. \]

\end{proof}
The proofs of Lemmas \ref{lem:wellcycle} and \ref{lem:wellperm} are omitted since they are in \cite{LWW}. 
\begin{Lem}\label{lem:wellcycle}
$(x;v)$ is a well-defined cycle with $r$ $q_1$-steps and $s$ $q_2$-steps. 
\end{Lem}

Now let $\gcd(r,s,l) = k > 1$. Consider $(x;v)$ where $v$ is determined by the lattice path $\nu$ from $(0,0)$ to $(r/k,s/k)$ constructed at the beginning of the section. Lemma \ref{lem:wellcycle} tells us that $(x;v)$ is a valid cycle. 

\begin{Lem}\label{lem:wellperm}
Let $k = \gcd(r,s,l)$ and $\nu$ be as above and then define $C_j$ = $(1 + (j - 1)(q_2 - q_1); v)$. Then 
\[\sigma = C_1 C_2 \cdots C_k\]
is a well-defined element of $T_{p;q_1,q_2}(r,s)$. 
\end{Lem}

Combining the results of Lemmas \ref{lem:idcycle} and \ref{lem:wellperm}, and again recalling the sign convention $\Phi=1-\det(C_\Phi)$ now gives the result summarized in Theorem \ref{thm:coefsign}.

\section{Results for $U(1,1)$} \label{s:general}

For the more general case of $U(1,1)$ we turn back to the more general cyclic group $\Gamma_{p;q_1,q_2}$ generated by
$$
s= \begin{bmatrix}
w^{q_1} && 0\\
0 && w^{q_2}\\
\end{bmatrix}
$$
Here again, we are restricting ourselves to the case of minimal $q_1,q_2$. For this case, the invariant polynomial gives a slightly different expression; namely,
$$
\Phi_{p;q_1,q_2}(z,\bar{z})= 1- \prod_{\gamma}(1-\langle \gamma z,z \rangle_1)=1-\prod_{j=1}^{p}(1-w^{jq_1} |z_1|^2 + w^{jq_2}|z_2|^2)
$$
As given by Theorem \ref{thm:coefsign} the coefficients in a related polynomial are subject to a number of conditions. For the polynomial:
$$
\Phi_{p;q_1,q_2}(z,\bar{z})= 1- \prod_{\gamma}(1-\langle \gamma z,z \rangle)=1-\prod_{j=1}^{p}(1-w^{jq_1} |z_1|^2 - w^{jq_2}|z_2|^2)
$$
 The coefficients $C_{rs}|z_1|^{2r}|z_2|^{2s}$ that are non-zero are exactly those where $p$ divides $ra+sb$, and $C_{rs}$ is positive if and only if $\gcd(r,s,\frac{ra+sb}{p})$ is odd. We can then use the transformation: $$|z_2|^2=y,\quad y \to -y$$
 to adapt Theorem \ref{thm:coefsign} for the case of $U(1,1)$.

\begin{Lem}\label{lem:weightineq}
Defining $N_l$ as the number of terms of weight $l$ the following inequalities hold
$$\left| N_l(\Gamma_{p;q_1,q_2})-\frac{lp}{q_1 q_2}\right| \leq 1 $$ when $1\leq l \leq q_1$ and 
$$\left| N_l(\Gamma_{p;q_1,q_2})-\frac{(q_2-l)p}{q_2(q_2-q_1)}\right| \leq 1 $$ when $q_1+1\leq l \leq q_2$.
\end{Lem}
\begin{proof} Fix $p, q_1, q_2$. We want to count the number of non-negative integer solutions $(r, s)$ such that $rq_1+sq_2 = lp$
and $r+s \leq p$ where $1 \leq l \leq q_2$. Note that $r = \frac{lp−sq_2}{q_1}$, so $r$ is an integer for one out of every $q_1$ values of $s$. Also
note that the two lines $rq_1 + sq_2 = lp$ and $r + s = p$ intersect at the point: $$r=\frac{(q_2-l)p}{q_2-q_1},\quad s=\frac{(l-q_1)p}{q_2-q_1}.$$
We can then split the proof into two cases: $1 \leq l \leq q_1$ and $q_1+1 \leq l \leq q_2$. In the case that $1 \leq l \leq q_1$ we have the intersection point occurring at a negative value of $s$, so we start the count at $s=0$. The number of non-zero terms is then within 1 of
$$\frac{\frac{lp}{q_2}-0}{q_1}=\frac{lp}{q_1q_2}.$$
 In the case where $q_1+1 \leq l \leq q_2$ the number of non-zero terms is within 1 of
$$\frac{\frac{lp}{q_2}-\frac{(l-q_1)p}{q_2-q_1}}{q_1}=\frac{(q_2-l)p}{q_2(q_2-q_1)}.$$
\end{proof}
\begin{Lem}\label{lem:evenodd}
Define $N^{\text{odd}}$ and $N^{\text{even}}$ as the number of odd and even weight terms, respectively. When $p$ is large, the asymptotic behavior is given by
$$N(\Gamma_{p;q_1,q_2}) \sim \frac{p}{2}.$$ Furthermore, $$\lim_{q_1 \to \infty} \lim_{q_2 \to \infty} \lim_{p \to \infty} \frac{N^{\text{odd}}(\Gamma_{p;q_1,q_2})}{N(\Gamma_{p;q_1,q_2})}=\lim_{q_1 \to \infty} \lim_{q_2 \to \infty} \lim_{p \to \infty} \frac{N^{\text{even}}(\Gamma_{p;q_1,q_2})}{N(\Gamma_{p;q_1,q_2})}.$$
\end{Lem}
\begin{proof} Using the results from the previous lemma, we can sum the number of terms of each weight $l$ for $1\leq l\leq q_2$ to approximate the ratios. Thus $N(\Gamma_{p;q_1,q_2})$ is within $q_2$ of 
$$
 \sum_{l=1}^{q_1}\frac{lp}{q_1q_2} + \sum_{l=q_1+1}^{q_2}\frac{(q_2-l)p}{q_2(q_2-q_1)}=\frac{p}{2}.
$$ 
Similarly we can approximate the numbers of even and odd weight terms. The exact sums depend on the parity of $q_1$ and $q_2$, but the asymptotic behavior of the ratios $N^{+}/N$ and $N^{-}/N$ is the same in each case. We illustrate the computation with $q_1$ and $q_2$ both odd below. By Lemma 
$$
N^{\text{even}}(\Gamma_{p;q_1,q_2}) \sim \sum_{l=1}^{ \frac{q_1-1}{2} }\frac{2lp}{q_1q_2} + \sum_{l= \frac{q_1+1}{2} }^{ \frac{q_2-1}{2} }\frac{(q_2-2l)p}{q_2(q_2-q_1)} = \frac{p ( q_1 q_2-1)}{4 q_1 q_2}.
$$$$
N^{\text{odd}}(\Gamma_{p;q_1,q_2}) \sim \sum_{l=0}^{ \frac{q_1-1}{2} }\frac{(2l+1)p}{q_1q_2} + \sum_{l= \frac{q_1+1}{2} }^{ \frac{q_2-1}{2} }\frac{(q_2-(2l+1))p}{q_2(q_2-q_1)} = \frac{p ( q_1 q_2+1)}{4 q_1 q_2}.
$$

Thus, in the limit where $p,q_1,q_2$ go to infinity, there will be an equal number of even and odd weight terms.
\end{proof}

\begin{proof}[Proofs of Theorems 3 and 4]
      To complete the proof of Theorem \ref{thm:asymp}, we break down the possible values of $p,q_1,q_2$ into 4 cases.
\begin{itemize}
\item Case 1: $q_1$ odd and $q_2$ odd. We are looking at $r,s$ such that $rq_1+sq_2=lp$. Applying the transformation $|z_2|^2 \to -|z_2|^2$ allows us to apply results from Theorem \ref{thm:coefsign}. We then know that a term is positive if and only if $s+\gcd(r,s,l)$ is odd. For $l$ even, $r$ and $s$ must be even also, leading to all negative terms. For odd weight $l$, half the terms have even $s$ and are thus negative, and half of the terms have odd $s$ and are thus positive. In summary, we are left with only half the odd weight terms as positive, so with Lemmas \ref{lem:weightineq} and \ref{lem:evenodd} we can then calculate the asymptotic positivity ratio. By Lemma \ref{lem:evenodd}, we have
$$
N^{\text{odd}}(\Gamma_{p;q_1,q_2})\sim \sum_{l=0}^{ \frac{q_1-1}{2} }\frac{(2l+1)p}{q_1q_2} + \sum_{l= \frac{q_1+1}{2} }^{ \frac{q_2-1}{2} }\frac{(q_2-(2l+1))p}{q_2(q_2-q_1)} = \frac{p (q_2q_1+1)}{4 q_1 q_2}.
$$
Since all even weight terms are negative and odd weight terms alternate in sign, we have $\frac{N^{\text{+}}}{N}=\frac{N^{\text{odd}}}{2N}$. Furthermore, by Lemma \ref{lem:evenodd}, we have $\frac{N^{+}}{N}\sim \frac{N^{\text{odd}}}{p}$, which gives  $$\lim_{p \to \infty} \frac{N^+(\Gamma_{p;q_1,q_2})}{N(\Gamma_{p;q_1,q_2})}=\frac{(q_2q_1+1)}{4 q_1 q_2}.
$$
\item  Case 2: $q_1$ odd and $q_2$ even.  This case proceeds exactly analogous to case 1, giving half the odd weight terms as positive. Using Lemmas \ref{lem:weightineq} and \ref{lem:evenodd} we can then get:
$$
N^{\text{odd}}(\Gamma_{p;q_1,q_2})\sim \sum_{l=0}^{ \frac{q_1-1}{2} }\frac{(2l+1)p}{q_1q_2} + \sum_{l= \frac{q_1+1}{2} }^{ \frac{q_2}{2}-1 }\frac{(q_2-(2l+1))p}{q_2(q_2-q_1)} = \frac{p ( q_1 (q_2-q_1)+1)}{4 q_1 (q_2-q_1)}.
$$ Since all even weight terms are negative and odd weight terms alternate in sign, we have $\frac{N^{\text{+}}}{N}=\frac{N^{\text{odd}}}{2N}$, and hence the asymptotic positivity ratio is again given by $\frac{N^{\text{odd}}}{2N}\sim \frac{N^{\text{odd}}}{p}$, which gives $$\lim_{p \to \infty} \frac{N^+(\Gamma_{p;q_1,q_2})}{N(\Gamma_{p;q_1,q_2})}=\frac{q_1(q_2-q_1)+1}{4q_1(q_2-q_1)}.
$$

\item Case 3: $p$ odd, $q_1$ even and $q_2$ odd. In the case that $q_2$ is odd, it is important to distinguish the parity of the order $p$. For this case, we can follow a similar argument from Theorem \ref{thm:coefsign} to get half the even weight terms as positive. Using Lemmas \ref{lem:weightineq} and \ref{lem:evenodd} we then get
$$
N^{\text{even}}(\Gamma_{p;q_1,q_2})\sim \sum_{l=1}^{ \frac{q_1}{2} }\frac{2lp}{q_1q_2} + \sum_{l= \frac{q_1}{2}+1 }^{ \frac{q_2-1}{2} }\frac{(q_2-2l)p}{q_2(q_2-q_1)} = \frac{p (q_2(q_2-q_1)+1)}{4 q_2( q_2-q_1)}.
$$
Since all odd weight terms are negative and even weight terms alternate in sign, we have $\frac{N^{\text{+}}}{N}=\frac{N^{\text{even}}}{2N}$. Thus the asymptotic positivity ratio is given by $\frac{N^{\text{even}}}{2N}\sim \frac{N^{\text{even}}}{p}$, and hence $$\lim_{p \to \infty} \frac{N^+(\Gamma_{p;q_1,q_2})}{N(\Gamma_{p;q_1,q_2})}=\frac{q_2(q_2-q_1)+1}{4 q_2( q_2-q_1)}.
$$
\item Case 4: $p$ even, $q_1$ even and $q_2$ odd. We are looking at $r,s$ such that $rq_1+sq_2=lp$. From Theorem \ref{thm:coefsign} we know that a term will be positive if and only if $s+\gcd(r,s,l)$ is odd. For odd $l$, $s$ must be even, so all the terms are positive. For even $l$, $s$ must still be even, but now the terms with even $r$ are negative. So we get half the even weight terms are negative. Using Lemmas \ref{lem:weightineq} and \ref{lem:evenodd} we then get
$$
N^{\text{even}}(\Gamma_{p;q_1,q_2})\sim \sum_{l=1}^{ \frac{q_1}{2} }\frac{2lp}{q_1q_2} + \sum_{l= \frac{q_1}{2}+1 }^{ \frac{q_2-1}{2} }\frac{(q_2-2l)p}{q_2(q_2-q_1)} = \frac{p (q_2(q_2-q_1)+1)}{4 q_2( q_2-q_1)}.
$$
Since all odd weight terms are positive and even weight terms alternate in sign, we have $\frac{N^{\text{+}}}{N}=1-\frac{N^{\text{even}}}{2N}$. Thus the asymptotic positivity ratio is given by $(1-\frac{N^{\text{even}}}{2N})\sim (1-\frac{N^{\text{even}}}{p})$, and hence this gives $$\lim_{p \to \infty} \frac{N^+(\Gamma_{p;q_1,q_2})}{N(\Gamma_{p;q_1,q_2})}=1-\frac{q_2(q_2-q_1)+1}{4 q_2( q_2-q_1)}=\frac{3q_2(q_2-q_1)-1}{4q_2(q_2-q_1)}.
$$

\end{itemize}
For the proof of Theorem \ref{thm:sphere}, we can proceed similarly, but now with only 3 cases.
\begin{itemize}
\item Case 1: $q_1$ odd and $q_2$ even. Now we can proceed directly from Theorem \ref{thm:coefsign}. Recall that $rq_1+sq_2=lp$ and a term is positive if and only if $\gcd(r,s,l)$ is odd. For odd weight $l$, all the terms are positive. For even weight $l$, the terms with odd $s$ are positive and even $s$ are negative. So we end up with half the even weight terms as negative. Using Lemmas \ref{lem:weightineq} and \ref{lem:evenodd} we then get

$$
N^{\text{even}}(\Gamma_{p;q_1,q_2})\sim \sum_{l=1}^{ \frac{q_1-1}{2} }\frac{2lp}{q_1q_2} + \sum_{l= \frac{q_1+1}{2}
}^{ \frac{q_2}{2} }\frac{(q_2-2l)p}{q_2(q_2-q_1)} = \frac{p (q_1(q_1-q_2)+1)}{4 q_1( q_1-q_2)}.
$$
Since all odd weight terms are positive and even weight terms alternate in sign, we have $\frac{N^{\text{+}}}{N}=1-\frac{N^{\text{even}}}{2N}$. Since the asymptotic positivity ratio is given by $(1-\frac{N^{\text{even}}}{2N})\sim (1-\frac{N^{\text{even}}}{p})$, this gives $$\lim_{p \to \infty} \frac{N^+(\Gamma_{p;q_1,q_2})}{N(\Gamma_{p;q_1,q_2})}=1-\frac{ (q_1(q_1-q_2)+1)}{4 q_1( q_1-q_2)}=\frac{3q_1(q_2-q_1)+1}{4q_1(q_2-q_1)}.
$$
\item Case 2: $q_1$ odd and $q_2$ odd. This case proceeds analogous to case 1, where half the even weight terms are negative. Using Lemmas \ref{lem:weightineq} and \ref{lem:evenodd}, we then get
$$
N^{\text{even}}(\Gamma_{p;q_1,q_2}) \sim \sum_{l=1}^{ \frac{q_1-1}{2} }\frac{2lp}{q_1q_2} + \sum_{l= \frac{q_1+1}{2}}^{ \frac{q_2-1}{2} }\frac{(q_2-2l)p}{q_2(q_2-q_1)} = \frac{p(q_1 q_2-1)}{4q_1 q_2}.
$$
Since all odd weight terms are positive and even weight terms alternate in sign, we have $\frac{N^{\text{+}}}{N}=1-\frac{N^{\text{even}}}{2N}$. Since the asymptotic positivity ratio is again given by $(1-\frac{N^{\text{even}}}{2N})\sim (1-\frac{N^{\text{even}}}{p})$, this gives $$\lim_{p \to \infty} \frac{N^+(\Gamma_{p;q_1,q_2})}{N(\Gamma_{p;q_1,q_2})}=1-\frac{q_1 q_2-1}{4q_1 q_2}=\frac{3q_1 q_2+1}{4q_1 q_2}.
$$
\item Case 3: $q_1$ even and $q_2$ odd. This case proceeds analogous to cases 1 and 2, where half the even weight terms are negative. Using Lemmas \ref{lem:weightineq} and \ref{lem:evenodd} we  then get
$$
N^{\text{even}}(\Gamma_{p;q_1,q_2})\sim \sum_{l=1}^{ \frac{q_1}{2} }\frac{2lp}{q_1q_2} + \sum_{l= \frac{q_1}{2}+1 }^{ \frac{q_2-1}{2} }\frac{(q_2-2l)p}{q_2(q_2-q_1)} = \frac{p(q_2(q_1-q_2)-1)}{4q_2(q_1- q_2)}.
$$
Since all odd weight terms are positive and even weight terms alternate in sign, we have $\frac{N^{\text{+}}}{N}=1-\frac{N^{\text{even}}}{2N}$. Since the asymptotic positivity ratio is again given by $(1-\frac{N^{\text{even}}}{2N})\sim (1-\frac{N^{\text{even}}}{p})$, this gives $$\lim_{p \to \infty} \frac{N^+(\Gamma_{p;q_1,q_2})}{N(\Gamma_{p;q_1,q_2})}=1-\frac{q_2(q_1-q_2)-1}{4q_2(q_1- q_2)}=\frac{3q_2(q_2-q_1)-1}{4q_2(q_2-q_1)}.$$
\end{itemize}

\end{proof}


\begin{thebibliography}{KeSt2}
\bib{BER}{book}{
   author={Baouendi, M. Salah},
   author={Ebenfelt, Peter},
   author={Rothschild, Linda Preiss},
   title={Real submanifolds in complex space and their mappings},
   series={Princeton Mathematical Series},
   volume={47},
   publisher={Princeton University Press, Princeton, NJ},
   date={1999},
   pages={xii+404},
   isbn={0-691-00498-6},
   review={\MR{1668103}},
   doi={10.1515/9781400883967},
}

\bib{D1}{article}{
   author={D'Angelo, John P.},
   title={Invariant holomorphic mappings},
   journal={J. Geom. Anal.},
   volume={6},
   date={1996},
   number={2},
   pages={163--179},
   issn={1050-6926},
   review={\MR{1469120}},
   doi={10.1007/BF02921598},
}
\bib{D2}{article}{
   author={D'Angelo, John P.},
   title={The combinatorics of certain group-invariant mappings},
   journal={Complex Var. Elliptic Equ.},
   volume={58},
   date={2013},
   number={5},
   pages={621--634},
   issn={1747-6933},
   review={\MR{3170647}},
   doi={10.1080/17476933.2011.599117},
}
\bib{D3}{book}{
   author={D'Angelo, John P.},
   title={Several complex variables and the geometry of real hypersurfaces},
   series={Studies in Advanced Mathematics},
   publisher={CRC Press, Boca Raton, FL},
   date={1993},
   pages={xiv+272},
   isbn={0-8493-8272-6},
   review={\MR{1224231}},
}


\bib{DL}{article}{
   author={D'Angelo, John P.},
   author={Lichtblau, Daniel A.},
   title={Spherical space forms, CR mappings, and proper maps between balls},
   journal={J. Geom. Anal.},
   volume={2},
   date={1992},
   number={5},
   pages={391--415},
   issn={1050-6926},
   review={\MR{1184706}},
   doi={10.1007/BF02921298},
}
\bib{DX1}{article}{
   author={D'Angelo, John P.},
   author={Xiao, Ming},
   title={Symmetries in CR complexity theory},
   journal={Adv. Math.},
   volume={313},
   date={2017},
   pages={590--627},
   issn={0001-8708},
   review={\MR{3649233}},
   doi={10.1016/j.aim.2017.04.014},
}
\bib{DX2}{article}{
   author={D'Angelo, John P.},
   author={Xiao, Ming},
   title={Symmetries and regularity for holomorphic maps between balls},
   journal={to appear in Math Research Letters},
}
\bib{F1}{article}{
   author={Forstneri\v c, Franc},
   title={Proper holomorphic mappings: a survey},
   conference={
      title={Several complex variables},
      address={Stockholm},
      date={1987/1988},
   },
   book={
      series={Math. Notes},
      volume={38},
      publisher={Princeton Univ. Press, Princeton, NJ},
   },
   date={1993},
   pages={297--363},
   review={\MR{1207867}},
}
\bib{F2}{article}{
   author={Forstneri\v c, Franc},
   title={Proper holomorphic maps from balls},
   journal={Duke Math. J.},
   volume={53},
   date={1986},
   number={2},
   pages={427--441},
   issn={0012-7094},
   review={\MR{850544}},
   doi={10.1215/S0012-7094-86-05326-3},
}
\bib{F3}{article}{
   author={Forstneri\v c, Franc},
   title={Extending proper holomorphic mappings of positive codimension},
   journal={Invent. Math.},
   volume={95},
   date={1989},
   number={1},
   pages={31--61},
   issn={0020-9910},
   review={\MR{969413}},
   doi={10.1007/BF01394144},
}
\bib{G1}{article}{
   author={Grundmeier, Dusty},
   title={Signature pairs for group-invariant Hermitian polynomials},
   journal={Internat. J. Math.},
   volume={22},
   date={2011},
   number={3},
   pages={311--343},
   issn={0129-167X},
   review={\MR{2782691}},
   doi={10.1142/S0129167X11006775},
}
\bib{G2}{book}{
   author={Grundmeier, Dusty},
   title={Group-invariant CR mappings},
   note={Thesis (Ph.D.)--University of Illinois at Urbana-Champaign},
   publisher={ProQuest LLC, Ann Arbor, MI},
   date={2011},
   pages={74},
   isbn={978-1124-96540-6},
   review={\MR{2949803}},
}
\bib{GLV}{article}{
   author={Grundmeier, Dusty},
   author={Lebl, Ji\v r\'\i },
   author={Vivas, Liz},
   title={Bounding the rank of Hermitian forms and rigidity for CR mappings
   of hyperquadrics},
   journal={Math. Ann.},
   volume={358},
   date={2014},
   number={3-4},
   pages={1059--1089},
   issn={0025-5831},
   review={\MR{3175150}},
   doi={10.1007/s00208-013-0989-z},
}
\bib{LWW}{article}{
   author={Loehr, Nicholas A.},
   author={Warrington, Gregory S.},
   author={Wilf, Herbert S.},
   title={The combinatorics of a three-line circulant determinant},
   journal={Israel J. Math.},
   volume={143},
   date={2004},
   pages={141--156},
   issn={0021-2172},
   review={\MR{2106980}},
   doi={10.1007/BF02803496},
}
\bib{S1}{book}{
   author={Simon, Barry},
   title={Orthogonal polynomials on the unit circle. Part 2},
   series={American Mathematical Society Colloquium Publications},
   volume={54},
   note={Spectral theory},
   publisher={American Mathematical Society, Providence, RI},
   date={2005},
   pages={i–xxii and 467–1044.},
   isbn={0-8218-3675-7 },
   review={\MR{2105088}},
}
\bib{RM}{book}{
year = {2006},
volume = {2},
journal = {Foundations and Trends in Communications and Information Theory},
title = {Toeplitz and Circulant Matrices: A Review},
doi = {10.1561/0100000006},
issn = {1567-2190},
number = {3},
pages = {155-239},
author = {Robert M. Gray}
}

\end{thebibliography}
\end{document}